\documentclass[11 pt]{amsart}
\usepackage{graphicx}
\usepackage{amsmath}
\usepackage[all]{xy}
\usepackage[portuguese, english]{babel}
\usepackage{amssymb,amsmath}
\usepackage{cite}
\usepackage{hyperref}
\usepackage{color}
\usepackage[utf8]{inputenc}
\usepackage{hhline}
\usepackage{fancyhdr, blindtext}

\newtheorem{theorem}{Theorem}[section]
\newtheorem{lemma}[theorem]{Lemma}
\newtheorem{prop}[theorem]{Proposition}

\newtheorem{corollary}[theorem]{Corollary}

\theoremstyle{definition}
\newtheorem{definition}[theorem]{Definition}

\theoremstyle{remark}
\newtheorem{remark}[theorem]{Remark}
\newtheorem*{ackname}{Acknowledgements}
\numberwithin{equation}{section}

\newcommand{\rank}{\operatorname{rank}}

\begin{document}

 \title[Central Configurations and Generic Finiteness]{New equations for central configurations and Generic Finiteness }
\author{Thiago Dias }

\address{Departamento de Matem\'atica, Universidade Federal Rural de Pernambuco - Rua Dom Manuel de Medeiros s/n, 52171-900, Recife, Pernambuco, Brasil}

\email{thiago.diasoliveira@ufrpe.br}

\

\begin{abstract}
We consider the  finiteness problem for central configurations of the $n-$body problem. We prove that, for $n\geq4$, there exists a  (Zariski) closed subset $B$ in the  mass space $\mathbb{R}^{n}$, such that if $(m_1,...,m_n) \in  \mathbb{R}^n\setminus B$, then there is a finite number of corresponding classes of $(n-2)-$dimensional central configurations for potential associated to a semi-integer exponent. Also, we obtain trilinear homogeneous polynomial equations of degree $3$ for central configurations of fixed dimension and, for each integer $k \geq 1$, we show that the set of  mutual distances associated to a $k-$dimensional central configuration is contained in a determinantal algebraic set.
\end{abstract}
\maketitle

\section{Introduction}

Suppose $x_1,...,x_n \in \mathbb{R}^d$  are  the coordinates of punctual particles with respective positive masses $m_1,...,m_n$. The vector $x=(x_1,...,x_n) \in \mathbb{R}^{dn}$ will be called  \emph{configuration}. 

Assume that the mutual distances $r_{ij}=\| x_i-x_j\|$, $1 \leq i < j \leq n$ are non-zero real numbers and consider a family of homogeneous potentials defined by: 
\begin{equation}
U_{a}(x)=\frac{1}{2a+2}{\displaystyle\sum_{i < j }}m_im_jr_{ij}^{2a+2},
\end{equation}
if $a \in \dfrac{\mathbb{Z}}{2}=\{a: 2a \in \mathbb{Z}\}$ and $a \neq -1$, or
\begin{equation}
U_{a}(x)={\displaystyle\sum_{ i < j }}m_im_j\log r_{ij},
\end{equation}
if $a=-1.$
The equations  of  motion of the $n$ particles are given by: 
\begin{equation}\label{eqN}
m_{j}\ddot{x}_{j}= \sum_{\begin{subarray}{c}i=1\\ i\neq j\end{subarray}}^{n}m_{j}m_{i}r_{ij}^{2a}(x_{i}-x_{j})=-\frac{\partial U_{a}}{\partial x_{j}}, \qquad j=1,...,n.
\end{equation}

The \emph{$n-$body problem} consists in  solving the system of equations $\eqref{eqN}$.  When $a=-\frac{3}{2}$ we have the \emph{Newtonian $n$-body Problem}. 
\begin{definition}\label{i3}The configuration $x$, is called \emph{central configuration  associated to the potential} $U_a$, where $2a \in \mathbb{Z}$,    if there exists $\lambda \neq 0$ such that
\begin{equation}\label{eqcc}
\sum_{i\neq j}m_{i}(x_{i}-x_{j})r_{ij}^{2a}+\lambda(x_j-c)=0, \qquad j=1,...,n,
\end{equation}
for which
$$c=\frac{1}{M}\left(m_1x_1+\cdots m_nx_n\right) \qquad \text{and} \qquad M=m_1+\cdots m_n \neq 0$$
are, respectively, the  \emph{center of mass} and the \emph{total mass}. \end{definition} 

We notice that in the case $a=0$ every configuration $x=(x_{1},...,x_n)$ is central,  with $\lambda=M$, and the system \eqref{eqN} is immediately integrated.

An excellent reference  about the role of the exponent $a$ in the discussion of central configurations is \cite{albouy2003paper}.
The set of central configurations is invariant under rotations, translations and dilations, therefore, it is natural to study this  up to these transformations. The problem of counting the solutions of the equations $ (\ref{eqcc}) $ is fundamental in Celestial Mechanics, since  the only known explicit  solutions to the $ n $-body problem are homographic orbits, and these orbits have central configurations  as initial conditions.  In  \cite{chazy1918certaines}, Chazy proposed the following problem, that appears in Smale's list \cite{smale1998mathematical} as:

``For any choice of positive real numbers $m_1,..,m_n$, is the number of central configurations finite? ''

\noindent This question also appears in \cite{wintner1941analytical}. Currently, it is known as Chazy-Wintner-Smale conjecture or finiteness problem.

The dimension of a configuration $x$, denoted by $\delta(x)$, is the dimension of the smallest affine space that contains the points  $x_1,...,x_n \in \mathbb{R}^{d}.$ If $\delta(x)=1,2 \text{ or } 3$, we say, respectively, that the  configuration $x$ is \emph{collinear}, \emph{planar} or \emph{spatial}. 

The dimension of a configuration is an invariant that plays a very important role in obtaining  results  on the finiteness  problem. For $a=-\frac{3}{2}$, Euler showed that, for each triple of positive masses, there exists an unique collinear central configuration of $3$ bodies, up to reordering  masses.  Moulton generalized this result for collinear configurations of $n$ bodies. See \cite{euler1767trium} and \cite{moulton1910straight}. Lagrange proved that for each choice of positive masses $m_{1}$, $m_{2}$ and $m_{3}$, the unique planar central configuration  with $3$ bodies is the equilateral triangle \cite{lagrange1772essai}. More generally, Saari proved that, if we  choose positive masses $m_1,...,m_n$, then the only associated $(n-1)-$dimensional  central configuration is the regular $(n-1)$-dimensional simplex \cite{saari1980role}.  

Recently, Moeckel and Hampton \cite{hampton2006finiteness} proved that there exists a finite number of planar central configurations with four bodies. Albouy and Kaloshin proved, in \cite{albouy2012finiteness}, the generic finiteness for planar central configurations in the five-body problem. 

Moeckel proved, in \cite{moeckel2001generic}, the generic finiteness for central configurations of dimension $\delta(x)=n-2$. In \cite{hampton2011finiteness}, Hampton and Jensen obtained an explicit condition in the masses for generic finiteness of spatial central configurations with $5$ bodies. This strengthened the generic finiteness result of Moeckel in this case.

We would like to emphasize that our inspiration is  Moeckel's article \cite{ moeckel2001generic}. Among other reasons, his work is important because it contains finiteness results obtained using  algebraic geometric tools like the fibre dimension Theorem for quasi-projective varieties.

In the present article, we use the Jacobian criterion to generalize the result of generic finiteness obtained in \cite{moeckel2001generic}  for $n-2$ dimensional central configurations  with semi-integer exponent of the potential. In addition, we prove that $x$ is a $k-$dimensional central configuration  if, and only if, the associated Cayley-Menger matrix has rank $k+2$. This allows us to show that the set of mutual distances associated to a $k-$dimensional central configuration is contained in a determinantal algebraic set. The new equations for  central configurations will be presented in Theorem \ref{gwf} and Proposition $\ref{newequation}$.

We remark that to extend Moeckel's approach  to generic finiteness for  potentials $U_a$, $2a \in \mathbb{Z}$, by a simple adaptation of the argument given in \cite{moeckel2001generic} it is necessary to prove that the resultant of the polynomials $g^{a}_{ij}$ and $F$ defined in section $4$ of this paper is not trivial for every $a$ with $2a \in \mathbb{Z}\setminus\{0\}$. When $a$ is a prime number, replacing the third roots of unity by $|2a|$-th roots of unity in the statement of  Proposition $5$ of \cite{moeckel2001generic}  and using the $|2a|$-th cyclotomic polynomial, we can adapt the proof of this Proposition. This allows to find a point that is not in the zero locus of the resultant. However,  in order to adapt the proof of Proposition $5$  to a general semi-integer potential $a$, we do need to know the coefficients of the $|2a|-$th cyclotomic polynomial and this is not an easy task in general. 

\begin{ackname} I wish to thank  Eudes Naziazeno and Marcelo Pedro for their helpful remarks, Eduardo Leandro for being the advisor on this work, the anonymous referee for the valuable comments and the Department of Mathematics at {\it Universidade Federal Rural de Pernambuco} for their assistance.
\end{ackname}
\section{New equations for central configurations}

 In this section, we obtain trilinear homogeneous polynomial equations of degree $3$ for central configurations of fixed dimension.

Consider a configuration $x=(x_{1},...,x_{n}) \in \mathbb{R}^{dn}$, and $E_j$ the $\mathbb{R}-$vector space generated by  the vectors $x_1-x_j,...,x_n-x_j$ in $\mathbb{R}^{d}$. 

Observe that, $\delta(x)=dim E_j \leq n-1$,  for all $j=1, \cdots, n$. Hence, we may suppose, without loss of generality, that $d=n-k$,  for some $k\geq 1$. Consider the inclusion map of $\mathbb{R}^{n-k}$ into $\mathbb{R}^{n}$:
$$\begin{array}{cccc}
\iota_n:&\mathbb{R}^{n-k}&\longrightarrow& \mathbb{R}^{n}\\
&(a_1,...,a_{n-k})&\longmapsto& (1,a_1,...,a_{n-k},0,...,0).
\end{array}$$
We can identify a configuration $x=(x_1,...,x_n)\in \mathbb{R}^{(n-k)n}$  with a configuration in $\mathbb{R}^{n^{2}}$
by defining $\iota_n(x)=(\iota_n(x_{1}),...,\iota_n(x_n)).$ It is easy to see that $\delta(x)=\delta(\iota_n(x))$. Moreover, since that $\iota_n$ is an injective affine transformation, $x$ and $\tilde{x}$
are equivalent up to symmetries  if, and only if, $\iota_n(x)$ and $\iota_n(\tilde{x})$ are too.

\begin{definition}We say that the $n \times n$ matrix
$$X= \left(\begin{array}{ccc}
    \iota_n(x_1)^t & \ldots & \iota_n(x_n)^t
    \end{array}\right)$$
is the \emph{associated matrix  to the configuration $x$}, and denote its determinant $|X|$ by $w(x)$. 
\end{definition}

 Notice that $\text{dim}E_j+1$ is precisely the number of linearly independent columns of the matrix $X$, so we have: 
\begin{equation}\label{a1}
\delta(x)= \text{rank}(X)-1.
\end{equation}

 Denote by $x_{i_1...i_{k}}$, where $1 \leq i_1<i_2<...<i_k \leq n$, the subconfiguration of $n-k$ bodies obtained from the configuration $x$ by removing the bodies $x_{i_1},...,x_{i_k}$. The $(n-k)\times(n-k)$ configuration matrix of $x_{i_1...i_k}$ will be denoted by
$X_{i_1...i_{k}}$, and we write $w(x_{i_1...i_{k}})=|X_{i_1...i_k}|$. Since the rank of $X$ is $n-k+1$, we have that $|X_{i_{1}...i_{k-1}}|\neq0$, for some choice of $i_1,...,i_{k-1}$.


Suppose that $x$ is a configuration of dimension $n-k$ with mass $m_1,...,m_n$ and consider $X$ its matrix configuration. Define the quantities
\begin{equation}\label{dd2}
S_{ij}=S_{ji}=r_{ij}^{2a}-r_{0}^{2a},~ \text {for } i\neq j, \text{ and }r_{0}^{2a}=M^{-1}\lambda,
\end{equation}
in which $\lambda$ is as in equation $\eqref{eqcc}$. It is easy to see that $x$ is a central configuration with masses $m_1,...,m_n$ if, and only if, satisfies the equations
\begin{equation}\label{c22}
\sum_{\begin{subarray}{c}i=1\\ i\neq j\end{subarray}}^{n}m_{i}S_{ij}(x_{i}-x_{j})=0,~j=1,...,n.
\end{equation}

Consider the $(n-(k+1))$-dimensional exterior product given by
\begin{equation*}
 v^{j}_{i_{1}...i_{k}}=(x_1-x_j)\wedge...\wedge (x_n-x_j),
\end{equation*}
in which the terms  $(x_j-x_j),(x_{i_{1}}-x_j),...,(x_{i_{k}}-x_j)$ were omitted. Taking the wedge product of \eqref{c22} with $v^{j}_{i_{1}...i_{k}}$, we obtain:
\begin{equation} \label{c33}
\sum_{\begin{subarray}{c}i=1\\ i\neq j\end{subarray}}^{n}m_{i}S_{ij}(x_{i}-x_{j})\wedge v^{j}_{i_{1}...i_{k}} =0.
\end{equation}
Define
\begin{equation}\label{xa2}
\Delta_{i_1...i_{k-1}} =(-1)^{\sum_{l=1}^{k-1}i_{l}}w(x_{i_{1}...i_{k-1}}).
\end{equation}

By the Cramer rule, we see that the vector of $\mathbb{R}^{n-k+2}$ with $j-$th coordinates $(-1)^{k_{j}}\Delta_{i_{1}...i_{k_j}ji_{k_{j+1}}...i_{k-2}}$, where
$j \in \{1,...,n\}\setminus \{i_{1},...,i_{k-2}\}$ and $k_j=\#\{i_{l}: i_{l}<j\}$, belongs to the kernel of the matrix $X_{i_{1},...,i_{k-2}}$. Using basic results of exterior algebra, we can prove that the equations \eqref{c33} implies the following Theorem:

\begin{theorem}\label{gwf}
If $x$ is a central configuration of dimension $n-k$ with masses $m_1,...,m_n$, then
\begin{equation}\label{c24}
\sum_{l=1}^{k}(-1)^{l}m_{i_{l}}S_{i_{l}j}\Delta_{i_1...i_{l-1}i_{l+1}...i_{k}}=0.
 \end{equation}

\end{theorem}

We notice that the last result  generalize the formulae obtained by Williams for central configurations of the five-body problem given by the equations $(5)$ of \cite{williams1938permanent} and provide trilinear homogeneous polynomial equations of degree $3$ for central configurations of fixed dimension that not depends explicitly of $x$.

The next result provides a polynomial parametrization for  $S_{ij}$ from a central configuration of dimension $n-2$. Proofs can be found in \cite{albouy2003paper} and \cite{moeckel2001generic}. 

 \begin{prop} \label{p3} Let $x$ be a central configuration with non-zero masses and  $\delta(x)=n-2$. If $(\Delta_{1},...,\Delta_n) \in \text{Ker}(X)$, define
\begin{equation}
w_{i}:=\frac{\Delta_i}{m_i}, \quad i=1,...,n,
\end{equation} then there exists a constant $\kappa \neq 0$ such that
\begin{equation}\label{2c25}
 S_{ij}=\kappa w_{i}w_{j}.
 \end{equation}
Moreover, at least two of the $w_{i}$'s are non-zero.
\end{prop}

\section{Determinantal formulas for central configurations}
In this section, we obtain a dimension criterion for central configurations that depends only of their $q=\frac{1}{2}n(n-1)$ mutual distances $r_{ij}$, $1\leq i <j \leq n$, and prove that the set of mutual distances associated to a $k-$dimensional central configuration is contained in a determinantal algebraic set.

The  \emph{Cayley-Menger matrix} associated to the vector $s=(s_{ij})\in \mathbb{C}^{q}$, $1\leq i < j\leq n$ is the $(n+1)\times (n+1)$ symmetric  matrix given by:
$$A(s)=\left(
\begin{array}{cccccc} \label{m.11}
0 &1 &1&1& \ldots  & 1 \\
1 & 0&s_{12}& s_{13}&\ldots & s_{1n} \\
1 & s_{12}&0& s_{23}&\ldots & s_{2n}\\
\vdots &\vdots &\vdots& \vdots & & \vdots\\
1 & s_{1n}&s_{2n}& s_{3n}&\ldots & 0\\
\end{array}
\right).$$
The \emph{Cayley-Menger determinant} is defined by $F(s)= \mid A(s) \mid$.

Index the rows and columns of $A(s)$ from $0$ to  $n$. 
Note that we can associate a configuration $x$ to a Cayley-Menger matrix $A(r)$, also denoted by $A(x)$, where $r=(r_{ij})\in \mathbb{C}^{q}$ is the vector of mutual distances between the points $x_1,...,x_n.$

\begin{prop} \label{moeck1}  
$x$ is a configuration of dimension $\delta(x)=n-2$ if and only if  $\rank(A(x))=n$.
\end{prop}
This Proposition is demonstrated in \cite{moeckel2001generic}.  As an immediate consequence, we have that the Cayley-Menger determinant vanishes at all  configurations of dimension $n-2$. Our objective in this section is to generalize this result.

\begin{lemma}\label{rem}
Let $x=(x_1,...,x_n)$ be a configuration, $X$ its configuration matrix and $A(x)$, the associated Cayley-Menger matrix. If the vectors
$$v_{1}=(v_{11},...,v_{1n}),...,v_{k}=(v_{k1},...,v_{kn})$$
are linearly independent and belong to the  kernel of $X$, then the vectors
$$\widetilde{v}_{1}=(-\sum_{i=1}^{n}\|x_{i}\|^{2}v_{1i},v_{11},...,v_{1n}),...,\widetilde{v}_{k}=(-\sum_{i=1}^{n}\|x_{i}\|^{2}v_{ki},v_{k1},...,v_{kn})$$
are also linearly independent and belong to the kernel of $A(x)$. In particular, $dim(Ker(X))\leq dim(Ker(A(x))).$
\end{lemma}
\begin{proof}
Firstly, observe that the  $\widetilde{v_{1}},...,\widetilde{v}_{k}$ belong to the kernel of $A$. In fact,   
\begin{equation*} 
\sum_{j=1}^{n}s_{ij}v_{ij}=\|x_{i}\|^{2}\sum_{j=1}^{n}v_{ij}-2x_{i}\cdot\sum_{j=1}^{n}v_{ij}x_{j}+\sum_{j=1}^{n}\|x_{j}\|^{2}v_{ij}=\sum_{j=1}^{n}\|x_{j}\|^{2}v_{ij}.
\end{equation*}
So, the vectors  $\widetilde{v_{1}},...,\widetilde{v}_{k}$ satisfy the linear equations of  $Ker(A(x))$.
A null linear combination of the vectors $\widetilde{v_{1}},...,\widetilde{v}_{k}$ allows us to obtain a null linear combination of the vectors $v_1,...,v_k$.
In this way, since $v_{1},...,v_{k} $ are linearly independent, so are $\widetilde{v}_{1},...,\widetilde{v}_{k}$. In particular, if $\{v_{1},...,v_{k}\}$ is a base for $Ker(X)$,
 then $\{\widetilde{v}_{1},...,\widetilde{v}_{k}\}$ are a linearly independent set contained in $\text{Ker}(A(x))$. Hence, $dim(\text{Ker}(x))\leq dim(\text{Ker}(A(x)))$.
\end{proof}

Denote by $F_{ij}(s)$ the cofactor of the entry $a_{ij}$ of $A(s)$. Now, we are ready to prove the main result of this section:

\begin{prop}\label{kp11}
The following conditions are equivalent:
\begin{enumerate}
\item $\delta(x)=n-k$, $k \geq 2$;
\item $\text{rank}(A(x))=n-k+2$.
\end{enumerate}
\end{prop}
\begin{proof}
We prove this result inductively. Proposition $\ref{moeck1}$ takes care of the case $k=2$. Suppose that the equivalence is proved for $n=2,...,k-1$. We will prove the desired equivalence for $n=k.$

 To prove that (1) implies (2), let $\delta(x)=n-k$. There exist indices $i_1,...,i_{k-2}$ such that the configuration of $n-k+2$ bodies $x_{i_1...i_{k-2}}$ has dimension $\delta(x_{i_1...i_{k-2}})=n-k$. By the Proposition $\ref{moeck1}$, the Cayley-Menger matrix $A(x_{i_1...i_{k-2}})$ has rank $n-k+2$. Hence, there exist $i_0,j_0 \in \{0,1,...,n\}$ such that  $F_{i_0j_0}(x_{i_1...i_{k-2}})$ is a non-zero minor of the matrix $A(x_{i_1...i_{k-2}})$. Observe that $F_{i_0j_0}(x_{i_1...i_{k-2}})$ is the determinant of a square submatrix $(n-k+2)\times(n-k+2)$ of the matrix $A$. Thus we have
$$\text{rank}(A(x))\geq n-k+2.$$
On the other hand, by Lemma $\ref{rem}$, $\text{dim}(\text{Ker}(A(x)))\geq \text{dim}(\text{Ker}(X))=k-1.$ Hence, $$\text{rank}(A(x))=n+1-\text{dim}(\text{Ker}(A(x)))\leq (n+1)-(k-1)=n-k+2.$$

 Next we prove that (2) implies (1). If $\text{rank}(A(x))=n-k+2$, then $\text{dim}(\text{Ker}(A(x)))$ is equal to $k-1$. By Lemma \ref{rem}, $\text{dim}(\text{Ker}(X))\leq k-1.$

Suppose that $\text{dim}(\text{Ker}(X))=l-1$, with $l<k$. From equation \eqref{a1}, we obtain $\delta(x)=n-l$. As $l<k$, by induction hypothesis $\text{rank}(A(x))=n-l+2$, which is a contradiction. Hence $l=k$ and $\text{rank}(A(x))=n-k+2$.
\end{proof}

\begin{corollary}
Let $x$ be a configuration. Then $Ker(A(x))$ and $Ker(X)$ are isomorphic through the linear map
$$\begin{array}{cccc}
 T:& Ker(X)&\rightarrow& Ker(A(x))\\
 &v=(v_1,...,v_n) &\mapsto& \widetilde{v}=(-\sum_{i=1}^{n}\|x_{i}\|^{2}v_{i},v_{1},...,v_{n})
\end{array}$$
\end{corollary}

Consider the polynomial ring $S=\mathbb{C}[R_{12},...,$ $R_{(n-1)n}]$. Let $A(R)$ be the Cayley-Menger matrix associated to de vector $R=(R^{2}_{12},...,$ $R^{2}_{(n-1)n}) \in S^{q}$. Therefore, every minor determinant of $A(R)$ is a polynomial in $S$.
Let $J_{k}$ be the ideal of $S$ generated by all the $k \times k$ minor determinants of $A(R)$. We denote by $N_{k}\subset \mathbb{C}^{\frac{n(n-1)}{2}}$, the determinantal algebraic set given by the zero locus of the ideal $J_{k}$. The next result provides determinantal formulae for  the  vectors of mutual distances associated to a central configuration.

\begin{prop}\label{newequation}
For each fixed $k \geq1$, let $r=(r_{12},...,r_{(n-1)n})\in \mathbb{R}^{q}\subset\mathbb{C}^{q}$ be a vector consisting of the mutual distances associated to a $k-$dimensional central configuration $x$. Every $(k+3)\times (k+3)$ minor determinant of $A(R)$ evaluated in $r$ is zero. Moreover, $r$ belongs to the quasi-affine algebraic set $N_{k+3}\setminus N_{k+2}$. 
\end{prop} 

\begin{proof} Given a central configuration $x$ of dimension $\delta(x)=k$ and $r=(r_{ij})\in\mathbb{C}^{q}$ its associated point of mutual distances, by Proposition \ref{kp11}, we have that $\text{rank}(A(x))=k+2$.
Therefore, every minor determinants of $A(R)$ with size $k+3$ evaluated in $r\in\mathbb{C}^{q}$ are zero. So $r \in N_{k+3}$. Moreover, if $r=(r_{12},...,r_{(n-1)n}) \in N_{k+2}$ then $rank(A(x))< n-2$, which is inconsistent with our claim. 
\end{proof}

Suppose that  $x$ is a central configuration of dimension $n-2$ with non-zero masses. Hence, the kernels of the matrices $A(x)$
and $X$ are isomorphic and have dimension $1$. If $(\Delta_{1},...,\Delta_{n})$ is the unique generator of $Ker(X)$, then  the unique generator of $Ker(A(x))$ is $(\Delta_0,\Delta_{1},...,\Delta_{n})$, where $\Delta_{0}=-\sum_{i=1}^{n}\|x_{i}\|^{2}\Delta_{i}$. We also know that every line of the cofactor matrix of $A(x)$ belongs to its kernel. Hence, there exist non-zero  constants $b_i$ such that $F_{ij}=b_{i}\Delta_{j}$. Since $A(x)$ is a symmetric $b_{i}\Delta_{j}=F_{ij}=F_{ji}=b_{j}\Delta_{i}$. In this way, we get that
$$\left(\begin{array}{ccc}
b_{0} &  \ldots  & b_{n} \\
\Delta_{0} & \ldots & \Delta_{n} \\
\end{array}\right)$$
has rank $1$. Since the second row of this matrix is non-trivial, there exists a unique $\alpha \neq 0$ such that $\alpha \Delta_{i}=b_{i}$. This proves the following result:

\begin{prop}\label{2z1} Suppose that $x$ is a central configuration of dimension $n-2$.
 If $(\Delta_{0},...,\Delta_{n})$ belongs to the kernel of $A(x)$, there exists an unique non zero constant $\alpha$ such that 
\begin{eqnarray} \label{2c14}
F_{ij}=\alpha\Delta_{i}\Delta_{j},~0\leq i,j \leq n.
\end{eqnarray}
Moreover, at least two $\Delta_{i}$'s are non zero real numbers.
\end{prop}

\section{The Algebraic Set of $(n-2)$-Dimensional Central Configurations}\label{s4}

In this section, we  define a family $\{V^{a}\}_{2a\in \mathbb{Z}\setminus \{0\}}$ of quasi-affine algebraic sets on $\mathbb{C}^{q+2n+1}$ such that  $V^{a}$ contains every point associated to an $(n-2)-$dimensional central configurations associated to the potential $U_{a}$, for every $a\in \frac{\mathbb{Z}}{2}\setminus \{0\}$.

Consider $x$  an $(n-2)-$dimensional central configuration with respective masses $m_1,...,m_n$. Up to an homothety on $x$ we  may consider $r_{0}=1$  in the equations $\eqref{dd2}$. This homothety fixes a representant of the class of $x$ modulo dilations.

Let $\mathfrak{X}^{a}$ be the set of all classes modulo symmetries of $(n-2)-$ dimensional central configurations with nonzero semi-integer exponent $a$  such that $r_0=1$. Toward a solution for finiteness problem, we can restrict our attention to the set $\mathfrak{X}^{a}$.

 We  use the notation $(r,z,\Delta_0,m)$ for an arbitrary point of $\mathbb{C}^{q+2n+1}$ with coordinates $(r_{12},...,r_{(n-1)n}, z_1,...,z_n, \Delta_0,m_1,...,m_n).$ 

From now to the end, $r_{ij}$ denotes a coordinate of a point in $P \in\mathbb{C}^{q+2n+1}$. When $P$ is associated to $x$, a $(n-2)-$dimensional central configuration, we write $P_x$ in place of $P$. In this case $r_{ij}$'s represent the mutual distances of $x$ and, in particular, are  positive real numbers.

 An association between points of $C^{q+2n+1}$ and elements of $\mathfrak{X}^{a}$ will be established in the theorem  \ref{newthm}. The ring of regular functions on $\mathbb{C}^{q+2n+1}$ will be denoted by $T=\mathbb{C}[R_{12},...,R_{(n-1)n},Z_1,...,Z_n,\Delta,M_1,...,M_n]$. 

Define the following algebraic set
\begin{equation*}
\widetilde{V}^{a}=\{(r,z,\Delta_0,m)\in \mathbb{C}^{q+2n+1}: g^{a}_{ij}=0,~F(R)=0,~\Gamma_l=0,~l=0,...,n\},
\end{equation*}
where  $F(R) \in T$ is the  determinant of the Cayley-Menger matrix $A(R)$, $\Gamma_0= \sum_{j=1}^n M_j Z_j$, $\Gamma_l = \Delta +  \sum_{j=1}^n M_j Z_j R_{jl}^{2}, \quad l=1,...,n$, and
$$g^{a}_{ij}=\begin{cases}
R_{ij}^{2a}-1-Z_{i}Z_{j},   & \text{ if } a>0, \\
R_{ij}^{-2a}(Z_iZ_j+1)-1,   & \text{ if } a<0,
\end{cases}$$
for which $1\leq i<j \leq n.$

\begin{theorem}\label{newthm} Let $x$ be an $(n-2)-$dimensional central configuration  with respective masses $m_1,...,m_n$ and semi-integer exponent $a\in \frac{\mathbb{Z}}{2}\setminus \{0\}$. There exists a multivalued function $f_{a}: \mathfrak{X}^{a}\rightarrow \widetilde{V}^{a}$ that associates $x$ to the points  $P_{x}=(r,z_1,...,z_n,\Delta_0,m)$ and $P^{*}_{x}=(r,-z_1,...,-z_n,-\Delta_0,m)$ $\in\widetilde{V}^{a}$, for which $r=(r_{ij})$ is the vector of mutual distances of $x$, $m=(m_i)$ is the vector of the respective mass of $x$ and  the $z_i$'s and $\Delta_{0}$ are determined by polynomial equations that define $\widetilde{V}^{a}$. Moreover,  the numbers $r_{ij}$ and $m_i$ are real and positive, the numbers $z_i$'s and $\Delta_{0}$ are all real or all pure imaginary and there exist at least two nonzero  $z_i$'s.
\end{theorem}
\begin{proof}
 Let $X$ be the configuration matrix of $x$. Consider the quantities $\Delta_{j}$ defined in the equations $\eqref{xa2}$ for the case $k=2.$ We have that $(\Delta_1,...,\Delta_n)\in \mathbb{R}^{n}$  
belongs to the kernel of $X$. Applying the Proposition $\ref{p3}$ we have that there exists quantities $w_1,...,w_n$ such that the mutual distances $r_{ij}$ of $x$ satisfies  the following equations 
\begin{equation} \label{eqdzjr}
r_{ij}^{2a}-1-\kappa w_{i}w_{j}=0, \quad 1\leq i< j\leq n.
\end{equation}
In order to eliminate the constant $\kappa$, we choose the principal branch of the complex square root function  and make a change of variables  $w \mapsto \sqrt{\kappa}z$, obtaining the following set of equations:
\begin{equation} \label{eqdzj}
r_{ij}^{2a}-1-z_{i}z_{j}=0, \quad 1\leq i< j\leq n. 
\end{equation}
 Note that if $\kappa>0$ then the quantities $z_j$ are all real number, and if $\kappa <0$ the quantities $z_j$ are all pure imaginary complex numbers.
 
 We remark that if we choose the other branch of the square root we obtain the same correspondence between $x$ and the points of $\widetilde{V}^{a}$. 

Observe that, if $z_1,...,z_n$ are such as in $\eqref{eqdzj}$, making 
\begin{equation}\label{eqdelta}
\Delta_{0}=-\sum_{j=1}^{n}\|x_{j}\|^2m_jz_j,
\end{equation} 
we have that the vector $(\Delta_{0},m_1z_{1},...,m_nz_{n})$, belong to the kernel of $A(x)$. Writing the equations for the kernel of $A(x)$ explicitly we obtain:
\begin{equation}\label{eqdzj2}
\sum_{j=1}^n m_j z_j=0, \quad
\Delta_0 +  \sum_{j=1}^n m_j z_j r_{jl}^{2}=0, \ \ l=1,...,n.
\end{equation}

Hence $x$ can be associated to a point $P_x=(r,z,\Delta_{0},m)$ where
$r$ is their vector of mutual distances, $m=(m_1,...,m_n)$ is the mass vector of $x$,  $z_1,...,z_n$  and $\Delta_0$ are all real  or all pure imaginary numbers determined by the equations of $\text{Ker}(A(x))$.

By the equations \eqref{eqdzj} and $\eqref{eqdzj2}$ then $g^a_{ij}(P_{x})=0$, $\Gamma_{0}(P_{x})=0$ and $\Gamma_{l}(P_x)$. Since $\delta(x)=n-2$, by the Proposition $\ref{moeck1}$ we have that $F(P_{x})=0.$ Hence, $P_x \in \widetilde{V}^{a}$.
Finally, if $x$ is associated to two points $P_{x}=(r,z,\Delta_0,m)$ and $P^{*}_{x}=(r,z^{*},\Delta^{*}_{0},m)$ of $\widetilde{V}^{a}$ then, since $\text{dim}(\text{Ker}(A(x)))=1$, for some $\xi \neq 0$, holds $(\Delta,z)=\xi(\Delta^{*},z^{*}).$ By the equations $\eqref{eqdzj2}$ there exists $z_{i_0}$ and $z_{j_0}$ both nonzero. Comparing the equations $g^{a}_{i_{0}j_{0}}(P(x))=0$ and $g^{a}_{i_{0}j_{0}}(P^{*}_{x})=0$ we get $\xi^{2}=1$. This proves the result.

\end{proof}


\begin{lemma} Let $F(R)$ be the determinant of the matrix $A(R)$ and $P_x=(r,z,\Delta_{0},m)$ a point associated to $x$, a $(n-2)-$dimensional central configuration. Then
\begin{equation}\label{eqnew}
\frac{\partial F}{\partial R_{ij}}(P_x)=4\alpha r_{ij}m_{i}z_{i}m_{j}z_{j}, \quad 1 \leq i<j\leq n.
\end{equation}
for which $\alpha$ is a constant defined as in Proposition $\ref{2z1}$,
\end{lemma}
\begin{proof}
 By the chain rule and the symmetry of the Cayley-Menger matrix, we have
  $$\frac{\partial F}{\partial R_{ij}}=2(R_{ij}F_{ij}+R_{ij}F_{ji})=4R_{ij}F_{ij}.$$
Since $P_x \in \widetilde{V}^{a}$, $(\Delta_{0},m_1z_1,...,m_nz_n)\in \text{Ker}(A(x))$, by Proposition $\ref{2z1}$, we get
\begin{equation*}\label{eqclas}
\frac{\partial F}{\partial R_{ij}}(P_x)=4r_{ij}F_{ij}(P_x)=4\alpha r_{ij}m_{i}z_{i}m_{j}z_{j}, \qquad 1 \leq i < j \leq n.
\end{equation*}
\end{proof}


\begin{prop}\label{crucial} Let $x$ be  a  $(n-2)$-dimensional central configuration  with semi-integer exponent  $a \in \frac{\mathbb{Z}}{2}\setminus\{0\}$ and $P_{x}\in \mathbb{C}^{q+2n+1}$ one associated point. Consider the polynomials
$$\Psi_{i}^{a}=\begin{cases}
~\quad R_{1i}^{-2a+1}\frac{\partial F}{\partial R_{1i}}Z_{1}+...+R_{ni}^{-2a+1}\frac{\partial F}{\partial R_{ni}}Z_{n},& \text{if } a<0, \\
{\textstyle\prod\limits_{(k,l)\neq (1,i)}}R_{kl}^{2a-1}\frac{\partial F}{\partial R_{1i}}Z_{1}+...+{\textstyle\prod\limits_{(k,l)\neq(1,n)}}R_{kl}^{2a-1}\frac{\partial F}{\partial R_{1n}}Z_{n},& \text{if } a>0,
\end{cases}$$
 i=1,...,n. Then, for every $a \in \frac{\mathbb{Z}}{2}\setminus\{0\}$, the equations $\Psi_{i}^{a}(P_{x})=0$  do not hold simultaneously.
\end{prop}
\begin{proof}
 We prove this result only in the case $a<0$, the other case is similar.  Using equations $\eqref{eqnew}$ we obtain, for each $i=1,\cdots,n$,
\begin{align*}
\Psi_i^{a}(P_{x})=&r_{1i}^{-2a+1}\frac{\partial F}{\partial R_{1i}}(P_x)z_{1}+ ...+ r_{ni}^{-2a+1}\frac{\partial F}{\partial r_{ni}}(P_x)z_{n} \\
=&4\alpha m_{i}z_i(m_{1}r_{12}^{-2a+2}z^{2}_1+...+m_{n}r_{ni}^{-2a+2}z^{2}_n).\\
\end{align*}
The quantities $z_i$ are all real or all pure imaginary and at least two of them are non-zero. Besides,  the masses and mutual distances are strictly  positive and the constant  $\alpha$ is non-zero. Hence, the equations $\Psi_{i}^{a}(P_{x})=0$ cannot  vanish simultaneously.
\end{proof}


 \begin{prop}\label{pnew}
 For every $a\in \frac{\mathbb{Z}}{2}\setminus\{0\}$, the quasi-affine algebraic set
 $$V^{a}=\widetilde{V}^{a}\setminus (W_1\cup W_2 \cup W^{a}_{3})$$
 contains every points $P_x$ associated to $(n-2)-$dimensional central configurations, where
\begin{align*}
&W_1=\{(r,z,\Delta_0,m) \in \mathbb{C}^{q+2n+1}:{\textstyle\prod\limits_{i<j}}  R_{ij}=0\},\\
&W_2=\{(r,z,\Delta_0,m) \in \mathbb{C}^{q+2n+1}:F_{ij}(R_{ij}^{2})=0\},\\
&W^{a}_{3}=\{(r,z,\Delta_0,m) \in \mathbb{C}^{q+2n+1}:\Psi^{a}_{i}=0, \quad i=1,...,n\}.
\end{align*}

\end{prop}
\begin{proof}
Consider $P_x \in f_{a}(\mathfrak{X}^{a})$, where $f_{a}$ as in Theorem $\ref{newthm}$. By construction of $f_{a}$, $P_{x}\in \widetilde{V}^{a}$. Note that $P_x$ not belongs to $W_1$ because the mutual distances $r_{ij}$ associated to a $(n-2)-$dimensional central configuration are non-zero. By Proposition  $\ref{newequation}$, applied in the case $k=n-2$,  $ P_x \not \in W_2$. Finally, as consequence of the Proposition $\ref{crucial}$, $P_x \not \in W^{a}_3$. This proves the result.
\end{proof}

Note that $V^{a}$ contains points of $\mathbb{C}^{q+2n+1}$  which are not associated to central configurations. We also denote the points of $V^{a}$ by $P$.

\section{The Jacobian Criterion}
 In this section we present the fundamental tool we use in order to calculate the dimension of $V^a$. The references for this section are \cite{shafarevich1994basic} and \cite{smith2000invitation}. 
 
 Let $W$ be an  affine algebraic set  with ideal $I(W)=\langle G_1,...,G_l\rangle$, where $G_1,...,G_l \in \mathbb{K}[Y_1,...,Y_n]$ and $\mathbb{K}$ is an algebraically closed field of zero characteristic. Let $Q$ be a point of $W$. The \emph{tangent space} of $W$ at $Q$ is the linear variety 
$$\Theta_{Q}W=Z(dG_{1}|_{Q}(Y-Q),...,dG_{l}|_{Q}(Y-Q))\subset \mathbb{A}_{\mathbb{K}}^{n},$$
where $\mathbb{A}_{\mathbb{K}}^{n}$ denotes de $n-$dimensional affine space over $\mathbb{K}$ and $Y=(Y_1,...,Y_n)$.
It can be proved that the tangent space is independent of the choice of generators for the ideal $I(W)$.

\begin{definition}
The \emph{dimension} of an  algebraic set $W$ at a point $Q \in W$, denoted by  $\text{dim}_{Q}W$, is the maximum of the dimensions of the irreducible components of $W$ that contain $Q$. A point $Q$ is said to be \emph{nonsingular} if $\text{dim}_{Q}(W)=\text{dim}(\Theta_{Q}W).$
\end{definition}

\begin{theorem}[Jacobian criterion]\label{cj}
Let $W$ be an affine algebraic set with ideal $I(W)=(G_{1},...,G_{l})$. Then $Q$ is a nonsingular  point of $W$ if and only if the rank of the Jacobian matrix $J(G_{1},...,G_l)(Q)$ of the ideal $I(W)$ at the point $Q$ is equal to $n-\text{dim}_{Q}(W)$, where, as usual,
$$J(G_{1},...,G_{l})(Q)= \left(
                    \frac{\partial{G_k}}{\partial Y_j}(Q)\right)_{1\leq k\leq l, \  
                    1\leq j \leq n}.$$
\end{theorem}
 
 It can be proved that $W$ is a quasi-projective algebraic set then 
 $$\text{dim}(\Theta_Q{W})\geq \text{dim}_{Q}(W),$$ 
 $\forall Q$ $\in W$ and it is easy to see that
  $\text{dim}W=\text{max}_{Q\in W}\text{dim}_{Q}W$.  This gives an upper bound for 
  $\text{dim(W)}$ in terms of the rank of the Jacobian matrix.

\section{The Jacobian Criterion and Generic Finiteness}

In \cite{moeckel2001generic}, Moeckel defined an algebraic set very similar to $V^{a}$ as in Proposition $\ref{pnew}$, for $a=-3/2$, and computed its dimension using the fibre dimension Theorem and resultants. In this section, by using Jacobian criterion, we prove that $\text{dim}V^{a}\leq n$ for each $a \in \frac{\mathbb{Z}}{2}\setminus\{0\}.$  This permits us to extend the generic finiteness result obtained in \cite{moeckel2001generic} for potentials with semi-integers exponents.  We work only in the case $a<0$, because the case  $a>0$ is similar. Remember that the case $a=0$ is trivial. For a point $P=(r,z,\Delta_0,m) \in V^{a}$, the Jacobian matrix $J(g^{a}_{ij}, F, \Gamma_{l} )(P)$  is the $(q+n+2)\times (q+2n+1)$ block matrix given by 
 $$\left(\begin{array}{l|l|l}
 \left[\frac{\partial g^{a}_{ij}}{\partial R}\scriptstyle(P)\right]_{q \times q \quad}& \left[\frac{\partial g^{a}_{ij}}{\partial Z }\scriptstyle(P)\right]_{q \times n ~\quad }&\left[\frac{\partial g^{a}_{ij}}{\partial (\Delta,M)}\scriptstyle(P)\right]_{q \times (n+1) ~ \quad}\\ \hhline{-|-|-}
 \left[\frac{\partial F}{\partial R}\scriptstyle(P)\right]_{1 \times q ~\quad } &\left[\frac{\partial F}{\partial Z }\scriptstyle(P)\right]_{1 \times n \qquad}& \left[\frac{\partial F}{\partial (\Delta,M)}\scriptstyle(P)\right]_{1 \times (n+1)}\\   \hhline{-|-|-}
\left[\frac{\partial \Gamma_{l}}{\partial R}\scriptstyle(P)\right]_{(n+1)\times q}& \left[\frac{\partial \Gamma_{l}}{\partial Z }\scriptstyle(P)\right]_{(n+1)\times n } &\left[\frac{\partial \Gamma_{l}}{\partial (\Delta,M)}\scriptstyle(P)\right]_{(n+1) \times (n+1)}
 \end{array}\right),$$
 where  $F(R)$ is the  determinant of the Cayley-Menger matrix $A(R)$, $\Gamma_0= \sum_{j=1}^n M_j Z_j$, $\Gamma_l = \Delta +  \sum_{j=1}^n M_j Z_j R_{jl}^{2}, \quad l=1,...,n$, and
$$g^{a}_{ij}=\begin{cases}
R_{ij}^{2a}-1-Z_{i}Z_{j},   & \text{ if } a>0, \\
R_{ij}^{-2a}(Z_iZ_j+1)-1,   & \text{ if } a<0,
\end{cases}$$
for which $1\leq i<j \leq n.$ 

 Now, we estimate the rank of $J(g^{a}_{ij}, F, \Gamma_{l} )(P)$  calculating the rank of its blocks.
 \begin{lemma} \label{jm1} Consider
\renewcommand{\arraystretch}{1.5}
 $$J(g^{a}_{ij},F)(P)=\left(\begin{array}{l|l}
 \left[\frac{\partial g^{a}_{ij}}{\partial R}\scriptstyle(P)\right]& \left[\frac{\partial g^{a}_{ij}}{\partial Z }\scriptstyle(P)\right]  \vspace{0.05cm}\\ \hhline{-|-}
 \left[\frac{\partial F}{\partial R}\scriptstyle(P)\right] &\left[\frac{\partial F}{\partial Z }\scriptstyle(P)\right]\\  
  \end{array}\right).$$
If $a<0$, for every point $P \in V^{a}$,  $J(g^{a}_{ij},F)(P)$ has a non-singular square submatrix $H(P)$ of order $q+1$.  In particular,  $J(g^{a}_{ij},F)(P)$ has maximal rank.
\end{lemma}
\begin{proof} Let $\delta_{ij}$ be the Kronecker's delta. Using the fact that $P \in V^{a}$ 
  we get $r_{ij}\frac{\partial g^{a}_{ij}}{\partial R_{kl}}(P)=\delta_{ik}\delta_{jl}(-2a)$, for which $1\leq i < j \leq n$ and $1\leq k<l \leq n$. Hence, $\left[\frac{\partial g^{a}_{ij}}{\partial R}(P)\right]$ is a diagonal matrix of entries $\frac{-2a}{r_{ij}}$. Moreover, $ \left[\frac{\partial g^{a}_{ij}}{\partial Z }(P)\right]$ is a matrix whose entries is given by $\frac{\partial g^{a}_{ij}}{\partial Z_{l}}=r_{ij}^{2a}(\delta_{il}z_{j}+\delta_{jl}z_{i})$,  $\left[\frac{\partial F}{\partial R}(P)\right]$ is the gradient of $F$ and $\left[\frac{\partial F}{\partial Z }(P)\right]$ is a null matrix.

By definition of $V^{a}$,  applying a linear isomorphism, if necessary, we have that every point $P\in V^{a}$ does not satisfy the equation:
  $$\Psi^{a}_1(P)=r_{12}^{-2a+1}\frac{\partial F}{\partial R_{12}}(P)z_{2}+...+r_{1n}^{-2a+1}\frac{\partial F}{\partial R_{1n}}(P)z_{n}=0.$$
  
By selecting the first $q+1$ rows and columns of  $J(g^{a}_{ij},F)(P)$, we get the  following  submatrix $(q+1) \times (q+1)$ :
 $$H(P)=\left(\begin{array}{cccccccccc}
\frac{-2a}{r_{12}}  & \ldots  & 0& \ldots &0&r_{12}^{-2a}z_{2} \\

\vdots   &\ddots & \vdots & \ldots & \vdots& \vdots            \\
0 & \ldots & \frac{-2a}{r_{1n}} &\ldots&0& r_{1n}^{-2a}z_{n}\\
0 & \ldots & 0                  & \ldots    &0&0\\  
\vdots&\ddots&\vdots&\ldots&\vdots&\vdots\\
0&\ldots&0&\ldots&\frac{-2a}{r_{(n-1)n}}&0\\
\frac{\partial F}{\partial R_{12}}\scriptstyle(P)&\ldots&\frac{\partial F }
{\partial R_{1n}}\scriptstyle(P)&\ldots& \frac{\partial F}{\partial R_{(n-1)n}}\scriptstyle(P)&0
\end{array}\right).$$

Applying the column operations:
   $$C_{q+1}\rightarrow \frac{1}{2a}r_{1(i+1)}^{-2a+1}z_{i+1}C_{i}+C_{q+1}, \quad i=1,...,n-1,$$
on $H(P)$, we  obtain
$$\tilde{H}(P)=\left(
\begin{array}{cccccccccccc}
\frac{-2a}{r_{12}}                 & \ldots  & 0& \ldots &0& 0             \\
\vdots                            &\ddots & \vdots&\ddots& \vdots&\vdots           \\
0                                 & \ldots & \frac{-2a}{r_{1n}}&\ldots&0&0\\
\vdots                            &\ddots&\vdots&\ddots&\vdots&\vdots\\
0                                 &\ldots&0&\ldots&\frac{-2a}{r_{(n-1)n}}&0\\
\frac{\partial F}{\partial R_{12}}&\ldots&\frac{\partial F }{\partial R_{1n}}& \ldots& \frac{\partial F}{\partial R_{(n-1)n}}&\gamma
\end{array}
\right),$$
where
$$\gamma=\frac{1}{2a}\left(r_{12}^{-2a+1}\frac{\partial F}{\partial R_{12}}(P)z_{2}+ ...+r_{1n}^{-2a+1}\frac{\partial F}{\partial R_{1n}}(P)z_{n}\right)=\frac{\Psi^{a}_1(P)}{2a}.$$
 The determinant of the matrix above is given by the product of the elements on the diagonal. Therefore it is given by $-(-2a)^{q-1}\Psi^{a}_{1}(P){\textstyle\prod} r_{ij}^{-1}\neq 0$, since $P \not \in W_{1}$. Hence $J(g^{a}_{ij},F)(P)$ has maximal rank.
 \end{proof}


Let $V^{a}=V_{1}\cup ... \cup V_{b}$ the decomposition of $V^{a}$ in irreducible components and consider the algebraic sets $D_{k}=\{(r,z,\Delta_{0},m) \in V^{a}: Z_k=0\},$ $k=1,...,n.$ We have two cases for consider: $V_i \not \subset D_k$, $\forall k$ or  $V_i \subset D_k$ for some $k$.

To compute de dimension of $V^{a}$ we need to calculate the rank of
\begin{equation}\label{jacayley}\left[\frac{\partial \Gamma_{l}}{\partial (\Delta,M)}(P)\right]=\left(
\begin{array}{cccccc} 
0 &z_1 &z_2&z_3& \ldots  & z_n \\
1 & 0&z_1r_{12}^2& z_2r^{2}_{13}&\ldots & z_{n}r^{2}_{1n} \\
1 & z_1r^{2}_{12}&0& z_2r^{2}_{23}&\ldots & z_nr^{2}_{2n}\\
\vdots &\vdots &\vdots& \vdots & & \vdots\\
1 & z_{1}r^{2}_{1n}&z_{2}r^{2}_{2n}& z_{3}r^{2}_{3n}&\ldots & 0\\
\end{array}
\right),
\end{equation}
in the two cases above.
\begin{lemma}\label{jm2}
Let $V_i$ be a irreducible component of $V^{a}$. If $V_{i} \not \subset D_{k}$, $\forall k$, there exists an open subset $U \in V_i$ such that for each $P\in U$,   $\left[\frac{\partial \Gamma_{l}}{\partial (\Delta,M)}(P)\right]$ has rank $n$.   In particular,  for each $P\in U$, $\left[\frac{\partial \Gamma_{l}}{\partial (\Delta,M)}(P)\right]$
has a non-singular square submatrix $G$ of order $n$.

\end{lemma}
\begin{proof}
 There exists an open subset $U \subset V_i$ such that if $P=(r,z,\Delta_{0},m) \in U$ then $z_i \neq 0$, $\forall i$. Note that, if $P\in U$, then $\left[\frac{\partial \Gamma_{l}}{\partial (\Delta,M)}(P)\right]$ and $A(r)$ are columns equivalent, so ${\textstyle\text{rank}\left(\left[\frac{\partial \Gamma_{l}}{\partial (\Delta,M)}(P)\right]\right)}=\text{rank}(A(r))$ which is equal to $n$, by Proposition \ref{moeck1}. \end{proof}

In order to estimate the dimension of the irreducible components of $V^{a}$, we remark that if $V$ is an irreducible algebraic set and $U$ is an  open subset in the induced topology on  $V$ then $\overline{U}=V$,  $\text{dim}U=\text{dim}V$ and $I(V)=I(V)$.
\begin{prop}\label{PT1}
 If $V_i \not \subset D_{k}$, $\forall k$, then $V_i$ has dimension at most $n$.
\end{prop}
\begin{proof}
Let $U \subset V_i$ be as in the Lemma $\ref{jm2}$. For $P \in U$, there exists, $M(P)$, a square submatrix of order 
$q+n+1$ of $J(g^{a}_{ij}, F, \Gamma_{l} )(P)$ such that:
\renewcommand{\arraystretch}{1.5}
 $$M(P)=\left(\begin{array}{c|c}
 H_{(q+1)\times (q+1)}&  0_{(q+1) \times n}\\ \hhline{-|-}
  Y_{n \times (q+1) } &G_{n \times n}  
  \end{array}\right)$$
  for which, $H$ as in the Lemma $\ref{jm1}$, $G$ is like in the Lemma $\ref{jm2}$, $0_{(q+1)\times n}$ is a null submatrix of $\left[\frac{\partial (g_{ij}^{a},F)}{\partial(\Delta,M)}\right]$ and $Y$ is a submatrix of $\left[\frac{\partial \Gamma_{l}}{\partial(R,Z)}\right]$. Since $M(P)$ is block triangular   and the matrices $H$ and $G$ are non-singular , $M(P)$ is non singular of order $n+q+1$.  Hence, by the Jacobian criterion, $\text{dim}_{P}(U)=\text{dim}_{P}(V_i)\leq n.$   This implies that $\text{dim}(U)=\text{dim}(V_i)\leq n$.\end{proof}

\begin{lemma}\label{jm4}
Let $V_{i}$ be an irreducible component of $V^{a}$ such that $V_i \subset D_k$ for some $k$, then there exists an open subset $U \in V_i$ and $t\in \{1,...,n\}$ such that for each $P\in U$, we have that $Z_1,...,Z_t \in I(U)$ and $n-t \leq {\textstyle\text{rank}\left(\left[\frac{\partial \Gamma_{l}}{\partial (\Delta,M)}(P)\right]\right)} \leq (n-t)+1.$ In particular,  for each $P\in U$, $\left[\frac{\partial \Gamma_{l}}{\partial (\Delta,M)}(P)\right]$
has a non-singular square submatrix $G$ of order $n-t$ or $n-t+1$.
\end{lemma}
 \begin{proof} In this case we can suppose, without loss of generality, that there exists unique $t\in\{1,...,n\}$ such that $V_{i}\subset \cap_{k=1}^{t}D_{k}$ and $V_{i}\not \subset D_{t+1},...,D_{n}$ .  There exists an open subset $U\subset V_i$ such that, for every $P \in U$, we have that $z_1=...=z_t = 0$ and $z_{t+1},...,z_n\neq0$. In particular $Z_1,...,Z_t \in I(U)$.  Let $v_0,v_1,...,v_{n}$ be the column vectors  of $\left[\frac{\partial \Gamma_{l}}{\partial (\Delta,M)}(P)\right]$ in \eqref{jacayley}. All the entries of the columns  $v_1,...,v_{t}$ are equal to zero. The columns  $v_0,v_{t+1},...,v_{n}$ are, up to rescaling with a non-zero constant, column vectors of the matrix $A(r)$. Since $A(r)$ has rank $n$, no more than one of the $n-t+1$ column vectors is a linear combination of the others, so the linear space generated by them has dimension $n-t$ or $n-t+1$. Hence we get the desired inequality.
\end{proof}

\begin{prop}
If  $V_{i}\subset D_{k}$, for some $k$,  then $\text{dim}V_i \leq n$.
\end{prop}
\begin{proof}
Let $U\subset V_i$ and $t\in\{1,...,n\}$ like in Lemma $\ref{jm4}$. The polynomials $g^{a}_{ij}$, $F$, $\Gamma_{l}$, $Z_{1},...Z_t$ belong to the ideal of $U$. Analogous to the proof of the Proposition $\ref{PT1}$, consider, for every $P \in U$, the following submatrix of $J(g^{a}_{ij}, F, \Gamma_{l}, Z_{1},...Z_t)(P)$:
\renewcommand{\arraystretch}{1.5}
 $$M(P)=\left(\begin{array}{c|c|c}
H_{(q+1) \times (q+1)}&  0_{(q+1)\times u}&0_{(q+1) \times t}\\ \hhline{-|-|-}
  Y_{u \times (q+1) } &G_{ u\times u }&0_{u \times t} \\ \hhline{-|-|-}
  0_{t\times (q+1)}&0_{t\times u}&I_{t\times t}  
  \end{array}\right)_{(q+n+1) \times (q+n+1)}$$
for which $H$ is as in the Lemma \ref{jm1}, $G$ is a square submatrix of $\left[\frac{\partial \Gamma_{l}}{\partial (\Delta,M)}\scriptsize{(P)}\right]$ of maximal rank and  $0_{t\times(q+1)}$, $0_{t\times u}$ and $I_{t \times t}$, the identity matrix of order $t$, are submatrices of $J(Z_1,...,Z_t)(P)$. Note that, by Lemma \ref{jm4}, the order of $G$ is $u=n-t \text{ or } (n-t)+1$.
Since $M(P)$ is block-diagonal, and the matrices $H,G$ and $I$ are non-singular, $M(P)$ is non singular. In this way, the rank of the Jacobian matrix $J(g^{a}_{ij}, F, \Gamma_{l}, Z_1,...,Z_t )(P)$ is greater than or equal to  $n+q+1$. So, by Jacobian criterion,  $\text{dim}_{P}(U)=\text{dim}_{P}(V_i)\leq n$ and, consequently, $\text{dim}(U)=\text{dim}(V_i)\leq n$.
\end{proof}

\begin{theorem}
$\text{dim}(V^a)\leq n$, for every $a\in \frac{\mathbb{Z}}{2}\setminus\{0\}$.
\end{theorem}
\begin{proof}
For every point $P \in V^a$, we have that $\text{dim}_{P}V_{i}\leq n$, for every component  $V_i$ of $V$ such that $P \in V_i$. 
Hence $\text{dim}_{P}V^{a} \leq n$ for every point $P \in V^a$. Since $\text{dim}V^{a}=\text{max}_{P\in V^{a}}\text{dim}_{P}V^{a}$, the result holds.
\end{proof}

Now, we will follow the last section in \cite{moeckel2001generic} in order to obtain the result of generic finiteness for central configurations with homogeneous semi-integer potentials.

Consider the projection $\pi: V^{a} \rightarrow \mathbb{C}^{n}$ such that $\pi(r,z,\Delta_0,m)=m$. Note that  $\pi$ is a regular morphism. If the restriction of $\pi$ to $V_i$ is a dominant map we say that $V_{i}$ is \emph{mass dominant}.

 \begin{theorem}\label{rp}
There exists an algebraic set $\widetilde{B} \subset \mathbb{C}^{n}$, such that if $m\in \mathbb{C}^{n}\setminus \widetilde{B}$, then the fibre $\pi^{-1}(m)$ is finite.
\end{theorem}
 \begin{proof}If $V_{i}$ is not  mass dominant, defining $B_i=\overline{\pi(V_i)}$ for every $m\in \mathbb{C}^{n}\setminus B_i$,  $\pi^{-1}(m) \cap V_i$ is empty.

If $V_j$ is mass dominant, then  $\text{dim} V_i \geq \text{dim} \mathbb{C}^{n}=n$. Hence $\text{dim}(V_j)=n$. Since the projection is regular, $V_j$ and $\mathbb{C}^n$ are  irreducible, by the fibre dimension Theorem, we have that there exists an closed set $A_j$ such that if $m \in \mathbb{C}^n\setminus A_j$, then $\pi^{-1}(m) \cap V_j$ is finite.

 Consider $\widetilde{B}=\big((\cup_{i}B_{i}) \cup (\cup_{j}A_j)\big)$. By choosing $m\in \mathbb{C}^{n} \setminus \widetilde{B}$ we have that the  fibre $\pi^{-1}(m)$ is finite.

 \end{proof}

\begin{theorem}
There exists a proper subvariety of the space of mass, $B \subset \mathbb{R}^{n}$, such that if $m\in \mathbb{R}^{n}\setminus B$, then for every $a \in \frac{\mathbb{Z}}{2}$, $m$ admits only a finite number of $(n-2)$-dimensional central configurations of dimension $\delta(x)=n-2$ with potential $U_a$ up to symmetry.
\end{theorem}
\begin{proof} If the proper subvariety $\widetilde{B}\subset \mathbb{C}^{n}$ is the algebraic set obtained in Theorem \ref{rp} then  we have that $B:=\widetilde{B}\cap\mathbb{R}^{n}\subset \mathbb{R}^{n}$ is a proper subvariety of $\mathbb{R}^{n}$, and if $m \in \mathbb{R}^{n} \setminus B$ the fibre $\pi^{-1}(m)$ is finite. This implies that for a fixed generic mass vector $m$,  we have a finite number of possibilities for the mutual distances  $r_{ij}$ associated to a $(n-2)-$dimensional central configurations. The mutual distances determine the configurations up to rotation and reflection. Hence, the Theorem is proved.
\end{proof}
\begin{remark}Let $x$ be a $(n-2)-$dimensional central configurations with at least one negative mass. Note that, using the same argument of the Proposition \ref{newthm},  we can associate $x$ to the points   $P_x$ and $P^{*}_{x} \in \widetilde{V}_{a}$.  However, in this case,  Propositions \ref{crucial} and \ref{pnew} are no longer true. In the other hand, if $\text{dim}_{P_{x}}(\widetilde{V}_a)>n$ then the fibre $\pi^{-1}(m)$ in $\widetilde{V}_{a}$ has dimension greater than or equal to $1$.
Hence, we believe that is possible to produce continuum of $(n-2)-$central configurations searching points $P_x\in W_3\cap \widetilde{V}_{a}$ and $\text{dim}_{P_{x}}(\widetilde{V}_a)>n$.
\end{remark}
Finally, we prove that for a generic  choice of masses the number of the correspondent $(n-2)-$dimensional central configurations
has an upper bound that is independent of the  choice of $m$.

Denote  a point $(r_{12},...,r_{(n-1)n},\kappa,w_1,...,w_n)\in \mathbb{R}^{q+n+1}$ by $(r,\kappa,w)$. Following again the last section of \cite{moeckel2001generic}, we use a Theorem provided by Thom and Milnor (\cite{Milnor} and \cite{Thom}): 

\begin{theorem}\label{MT}
Consider $f_1,...,f_m \in \mathbb{R}[x_1,...,x_n]$ and  define the set $X=\{x \in \mathbb{R}^{n}: f_1=...=f_m=0\}$. Then, the number of connected components   $X$ is lower or equal to $\beta(2\beta-1)^{n-1}$, for which $\beta=\text{max}\{\text{deg}(f_1),...,\text{deg}(f_m)\}$.
\end{theorem}

\begin{theorem}
If $n\geq 4$, for every $m=(m_1,...,m_n) \in \mathbb{R}^{n} \setminus B$, the number of $(n-2)-$dimensional central configurations with mass $m_1,...,m_n$ and potential $U_a$, $a \in \frac{\mathbb{Z}}{2}$ is less than or equal to $\beta(2\beta-1)^{q+n}$, where
$\beta$ is $\text{max}\{-2a+2,2n\}$ if  $a < 0$
and $\text{max}\{2a,2n\}$  if $a > 0$.

\end{theorem}
\begin{proof}We prove the Theorem for $a>0$. For the case $a<0$ the proof is analogous. Choose $m=(m_1,...,m_n)\in \mathbb{R}^{n}\setminus B$.
Every $(n-2)-$dimensional central configuration associated to $m$ belongs to the set
$$X_{m}=\{(r,\kappa,w)\in \mathbb{R}^{q+n+1}: F=0,~\tilde{g}^{a}_{ij}=0, ~\Omega_0=0,~\Omega_{i}-\Omega_{j}=0\},$$
for which $~1\leq i < j \leq n$, $\tilde{g}^{a}_{ij}=R_{ij}^{2a}-1-KW_iW_j$, $F$ is the determinant of the Cayley-Menger matrix $A(R)$, $\Omega_0=\sum_{k=1}^{n}m_kW_k$, $\Omega_i=\sum_{k=1}^n m_k W_k R_{kl}^{2}$  and the masses $m_1,...,m_n$ are constants. Note that $\tilde{g}^{a}_{ij}$, $F$, $\Omega_{0}$ and $\Omega_{i}$ belong to $\mathbb{R}[R_{12},...,R_{(n-1)n}, K, W_1,...,W_n]$. By Theorem \ref{newthm}, every point of $X_{m}$ is associated at  two points of $\pi^{-1}(m)$, so $X_{m}$ is finite. Observe that $\text{deg}(F)=2n$, $\text{deg}(\tilde{g}_{ij})=2a$, for every $a>0$, $1\leq i <j \leq n$, and $\text{deg}(\Omega_{j})\leq 3$ for every $ 0\leq j \leq n$. Since $n \geq 4$ implies $2n>2$, the maximum degree of the polynomials that define the set $X$ is $\beta=\text{max}\{2a,2n\}$. Observe that $\beta$ is independent of the choice of masses. By Theorem $\ref{MT}$, the number of connected components  of $X$ is at least $\beta(2\beta-1)^{q+n}$. Since $X_m$ is finite, every point of  $X$ determines your connected component. In particular,  every $(n-2)-$dimensional central configuration  associated to $m$ determines an unique connected component of $X_m$. This proves the result.
\end{proof}

\nocite{}
\bibliographystyle{alpha}
\bibliography{refs}

\end{document}